\documentclass[12pt]{article}
\usepackage{amsthm,amsmath,amssymb,latexsym,amscd}
\usepackage{graphicx}
\newcommand{\g}{\frak{g}}

\newcommand{\F}{\mathcal{F}}
\newcommand{\D}{\mathcal{D}}

\renewcommand{\P}{\mathcal{P}}

\newcommand{\Hom}{\mathrm{Hom}}

\newcommand{\End}{\mathcal{E}nd}

\newcommand{\Com}{\mathcal{C}om}

\newcommand{\Lie}{\mathcal{L}ie}
\newcommand{\Leib}{\mathcal{L}eib}

\newcommand{\Perm}{\mathcal{P}erm}
\newcommand{\Poiss}{\mathcal{P}oiss}
\renewcommand{\a}{\alpha}
\renewcommand{\b}{\beta}
\renewcommand{\d}{\mathbf{d}}
\newcommand{\p}{\prime}
\renewcommand{\c}{\circ}
\newcommand{\ot}{\otimes}
\newcommand{\s}{\mathbf{s}}

\newcommand{\vd}{\vdash}
\newcommand{\dv}{\dashv}

\newtheorem{definition}{Definition}[section]
\newtheorem{lemma}[definition]{Lemma}

\newtheorem{proposition}[definition]{Proposition}
\newtheorem{theorem}[definition]{Theorem}
\newtheorem{corollary}[definition]{Corollary}

\newtheorem{example}[definition]{Example}
\date{}

\begin{document}

\title{Derived bracket construction
and Manin products}

\author{K. UCHINO}
\maketitle
\abstract
{
We will extend the classical
derived bracket construction
to any algebra over a binary quadratic operad.
We will show that the derived product
construction is a functor given by the Manin white product
with the operad of permutation algebras.
As an application, we will show
that the operad of prePoisson algebras is isomorphic
to Manin black product of the Poisson operad with the
preLie operad.
We will show that differential operators
and Rota-Baxter operators are,
in a sense, Koszul dual to each other.
}

\section{Introduction.}

Let $(\g,[\cdot,\cdot],d)$ be a differential graded (dg) Lie algebra.
We define a new bracket product by
$[x,y]_{d}:=-(-1)^{|x|}[dx,y]$, $x,y\in\g$.
Then the new bracket becomes a Leibniz bracket
(so-called Loday bracket).
This method of constructing a new product
is called a \textbf{derived bracket}
construction of Koszul-Kosmann-Schwarzbach
(cf. Kosmann-Schwarzbach \cite{Kos1, Kos2}).
The derived bracket construction plays important roles
in modern analytical mechanics and in Poisson geometry.
It is known that several important brackets,
e.g., Poisson brackets, Schouten-Nijenhuis brackets,
Lie algebroid brackets, Courant brackets
and BV-brackets are induced by the
derived bracket construction.\\
\indent
The idea of the derived bracket construction arises in
several mathematical areas.
For instance, given a dg associative algebra $(A,d)$,
the modified products, $x\vd y:=-(-1)^{|x|}(dx)y$
and $x\dv y:=x(dy)$, are both associative.
The modified products,
which are called the \textbf{derived products},
are used in the study of Loday type algebras
(cf. Aguiar \cite{A}, Loday \cite{Lod2}).
\medskip\\
\indent
The purpose of this letter is to extend
the classical derived bracket/product construction
to any $\P$ (binary quadratic operad)-algebra .
It will be shown that the derived bracket
construction is a representation of the functor:
$$
\Perm\c(-):\P\mapsto\Perm\c\P,
$$
where $\Perm$ is the operad
governing permutation algebras (cf. Chapoton \cite{Chap})
and where $\c$ is the white product
of Manin-Ginzburg-Kapranov (so-called Main white product).
Let $\P$ be a binary quadratic operad
and let $A$ be a dg $\P$-algebra.
We regard the differential as a $1$-ary operation.
We define the derived brackets on $sA$ by
the usual manner:
\begin{eqnarray*}
\ [sx,sy]_{d}&:=&-(-1)^{|x|}s[dx,y],\\
\ [sx,sy]^{d}&:=&s[x,dy],
\end{eqnarray*}
where $s$ is a degree shifting operator
and the bracket is a multiplication of $\P$-algebra.
The derived brackets generate a new algebra structure on $sA$.
In \textbf{Theorem \ref{them1}},
we will show that the algebra generated by
the derived brackets is a $\Perm\c\P$-algebra.
To prove the theorem we apply
the detailed study for Manin products (Vallette \cite{Val}).
Vallette showed that the operad of Leibniz algebras,
$\Leib$, is isomorphic to $\Perm\c\Lie$ (Theorem 20 in \cite{Val}).
We will show that the theorem of Vallette
is the classical derived bracket construction
on the level of operad.
As an application, it will be shown that
the operad of prePoisson algebras (\cite{A}),
$pre\Poiss$, is isomorphic to the operad $pre\Lie\bullet\Poiss$,
where $pre\Lie$ is the operad of preLie algebras,
$\Poiss$ is the operad of Poisson algebras
and $\bullet$ is the Manin black product which
is the Koszul dual of the white product.\\
\indent
We will discuss a Koszul duality for
the derived bracket construction.
The differential operators satisfy the relation below.
$$
d[dx,y]=-[dx,dy]=-d[x,dy].
$$
This relation is the heart of the derived bracket construction.
We consider
the following operator identity on $\P^{!}$-algebras
instead of the derivation relation.
$$
\b(x)*\b(y)=\b(\b(x)*y+x*\b(y)),
$$
where $*$ is a multiplication of $\P^{!}$-algebra.
This $\b$ is called a Rota-Baxter operator.
It is known that some integral operators are Rota-Baxter operators.
We introduce a new concept,
the \textbf{dual-derived products},
which are multiplications
defined by the Rota-Baxter operators:
$$
\b(x)*y \ \ \text{and} \ \ x*\b(y).
$$
The original idea of the dual-derived products
was given by Aguiar in \cite{A}.
We will prove that the algebras generated by the
dual-derived products are $pre\Lie\bullet\P^{!}$-algebras
(\textbf{Theorem \ref{them2}} below).
Since $pre\Lie\bullet\P^{!}$ is the Koszul dual of $\Perm\c\P$,
the dual-derived product construction is a kind of
Koszul dual of the derived bracket construction.
Therefore, one can regard Rota-Baxter/integral operators as
Koszul dual operators of differential operators.
\medskip\\
\noindent
Acknowledgement.
The author would like to thank very
much referees for kind advice and useful comments
and also thank Professor Akira Yoshioka
for kind advice.
\section{Preliminaries}
\noindent
\textbf{Assumptions}.
The characteristic of the ground field $\mathbb{K}$ is zero.
We follow the standard Koszul sign convention.
\medskip\\
\indent
We recall the notion of (graded) operad.
For details,
see  Ginzburg-Kapranov \cite{GK,GK2} or
Markl-Shnider-Stasheff \cite{MSS}.
In Appendix, we give a short survey of Koszul duality
theory and Manin products.
\begin{example}
(endomorphism operad)
Let $V$ be a (graded) vector space.
We consider a collection of the spaces of linear endomorphisms,
$\End(V):=\{\End(V)(n)\}_{n\in\mathbb{N}}$,
$\End(V)(n):=\Hom(V^{\ot n},V)$.
For any $f\in \End(V)(m)$ and $g\in\End(V)(n)$,
we define an $i$th-composition by
$$
f\c_{i}g(x_{1},...,x_{m+n-1}):=
(-1)^{|g|(|x_{1}|+\cdot\cdot\cdot+|x_{i-1}|)}
f(x_{1},...,x_{i-1},g(x_{i},...,x_{i+n-1}),x_{i+n},...,x_{m+n-1}).
$$
where $|\cdot|$ is the degree of the object.
The composition is a binary map:
$$
\c_{i}:\End(V)(m)\ot\End(V)(n)\to\End(V)(m+n-1).
$$
The composition is associative in the sense
of (\ref{def1}) below.
\begin{eqnarray}\label{def1}
(f\c_{i}g)\c_{j}h=\left\{
\begin{array}{ll}
(-1)^{|g||h|}(f\c_{j}h)\c_{i+l-1}g & 1\le j\le i-1 \\
f\c_{i}(g\c_{j-i+1}h) & i\le j\le i+n-1 \\
(-1)^{|g||h|}(f\c_{j-n+1}h)\c_{i}g & i+n\le j \le m+n-1,
\end{array}
\right.
\end{eqnarray}
where $f\in\End(V)(m)$, $g\in\End(V)(n)$ and $h\in\End(V)(l)$.
The space $\End(V)(n)$ has a canonical $S_{n}$-module
structure defined by
$$
f\sigma(x_{1},...,x_{n}):=\epsilon(\sigma)
f(x_{\sigma^{-1}(1)},...,x_{\sigma^{-1}(n)}),
$$
where $\sigma\in S_{n}$ and $\epsilon(\sigma)$ is Koszul sign.
The composition is equivariant, that is,
\begin{equation}\label{def2}
f\sigma\c_{i}g\tau=(f\c_{\sigma^{-1}(i)}g)(\sigma\c_{i}\tau),
\end{equation}
where $\sigma\c_{i}\tau\in S_{m+n-1}$.
There exists the identity map $id:V\to V$ in $\End(V)(1)$.
The identity map is the unit element
with respect to the composition:
\begin{equation}\label{def3}
f\c_{i}id=f=id\c_{1}f.
\end{equation}
\end{example}
Let $E_{n}$ be a (graded) vector space and an $S_{n}$-module.
A collection of such spaces, $\{E_{n}\}_{n\in\mathbb{N}}$,
is called an $S$-module.
\begin{definition}
An operad is an $S$-module, $\{\P(n)\}_{n\in\mathbb{N}}$,
equipped with a collection of
compositions, $\{\c_{i}:\P(m)\ot\P(n)\to\P(m+n-1)\}$,
and the unit element $1\in\P(1)$
satisfying (\ref{def1}), (\ref{def2}) and (\ref{def3}).
\end{definition}
\begin{definition}
An operad morphism, $\phi:\P_{1}\to\P_{2}$, is
a collection of $S_{n}$-equivariant linear maps
of degree $0$,
$\{\phi(n):\P_{1}(n)\to\P_{2}(n)\}_{n\in\mathbb{N}}$,
which commute the operad composition maps
and which preserves the unit.
\end{definition}
\begin{example}(cf. \cite{GK} Example 2.1.10)
The operad of Lie algebras, $\Lie$, is an $S$-module
$\{\Lie(n)\}_{n\in\mathbb{N}}$
which is generated by a $1$-dimensional vector
space $\Lie(2):=\langle{\mu}\rangle$.
We assume $\Lie(1):=\mathbb{K}$.
The $S_{2}$-action on $\Lie(2)$ is given by
a sign representation (i.e. skewsymmetry).
The generating relation of the Lie operad
is the Jacobi identity,
$$
\mu\c_{2}\mu-\mu\c_{1}\mu-(\mu\c_{2}\mu)((12)\ot 1),
\ \ (12)\in S_{2},
$$
or explicitly,
$\mu(x_{1},\mu(x_{2},x_{3}))-\mu(\mu(x_{1},x_{2}),x_{3})-
\mu(x_{2},\mu(x_{1},x_{3}))$.
\end{example}
If $(\g,[\cdot,\cdot])$ be a Lie algebra, then
there exists a suboperad of the endomorphism operad
$\End(\g)$ which is generated
by the Lie bracket $[\cdot,\cdot]:\g\ot\g\to\g$.
This suboperad is a representation of the operad $\Lie$.
Conversely, given an operad morphism (representation)
$rep:\Lie\to\End(V)$,
$V$ becomes a Lie algebra.
\begin{definition}
Let $\P$ be an operad.
The $\P$-algebra structure is defined to be
an operad representation,
$$
rep:\P\to\End(V).
$$
\end{definition}
The notion of $\P$-algebra is defined by this way.
\begin{example}
(Leibniz algebras; $\Leib$; cf. \cite{Lod1,Lod2})
A (left) Leibniz algebra
(also called, a Loday algebra)
is a (graded) vector space
with a binary multiplication satisfying
the (left) Leibniz identity,
$$
[x,[y,z]]=[[x,y],z]+(-1)^{|x||y|}[y,[x,z]].
$$
When the bracket is skewsymmetric, it is a Lie algebra.
The operad of Leibniz algebras, $\Leib$, is generated by
a 2-dimensional vector space $\langle{\mu,\mu^{(12)}}\rangle$,
where $\mu^{(12)}$ is the transposition of $\mu$, that is,
$$
\mu^{(12)}(x_{1},x_{2})=\mu(x_{2},x_{1}).
$$
The generating relation of $\Leib$ is the same as
the relation of $\Lie$.
\end{example}
\begin{example}\label{exampleperm}
(Permutation-algebras; $\Perm$; cf. \cite{Chap,Chap2},
\cite{Val} Section 4.1)
A (graded) vector space with a binary
multiplication $(C,*)$ is called a (left) permutation
algebra (or called a perm-algebra for short),
if $*$ is associative and satisfies (\ref{relperm}) below.
\begin{equation}\label{relperm}
(x*y)*z-(-1)^{|x||y|}(y*x)*z=0,
\end{equation}
where $x,y,z\in C$.
One can check that the $n$th-component of
the operad of perm-algebras is
isomorphic to $\mathbb{K}^{n}$
for each $n\in\mathbb{N}$.
Therefore the dimension relation below holds.
$$
\dim\Perm(n)=n, \ \ \text{for each $n$}.
$$
We will use this relation in the next section.
We denote the linear basis of $\Perm(n)$ by
$\langle e^{n}_{1},...,e^{n}_{n}\rangle$.
Here $e^{n}_{j}$ corresponds with an $n$-monomial
whose right component is $x_{j}$:
$$
e^{n}_{j}\cong(x_{i_{1}}*\cdot\cdot\cdot*x_{i_{n-1}})*x_{j},
$$
where $i_{a}\in\{1,...,\check{j},...,n\}$. It is known that
the composition rule of $\Perm$ is given by
\begin{eqnarray*}
e^{m}_{i}\c_{j}e^{n}_{k}=\left\{
\begin{array}{ll}
e^{m+n-1}_{i}& i<j \\
e^{m+n-1}_{i+k-1}& j=i \\
e^{m+n-1}_{i+n-1}& i>j.
\end{array}
\right.
\end{eqnarray*}
We will use these relations in the next section.
\end{example}
\begin{example}
Let $\P_{1}$, $\P_{2}$ be operads.
Then $\P_{1}\ot\P_{2}:=\{\P_{1}(n)\ot\P_{2}(n)\}_{n\in\mathbb{N}}$
becomes an operad by a natural manner.
\end{example}
We recall the free operad.
\begin{example}
(cf. \cite{GK} Section 2.1.1, \cite{MSS} page 71)
Let $E:=(E_{1},E_{2},...)$ be an $S$-module.
We identify an element $f$ in $E_{n}$ with
the $n$-corolla
(a tree with one vertex and $n$-leaves)
whose vertex is decorated with $f$.
For instance, $f\in E_{3}$,
\begin{center}
\includegraphics[scale=0.4]{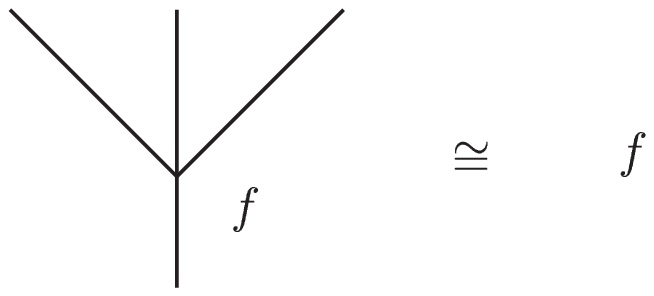}.
\end{center}
Given $f\in E(m)$, $g\in E(n)$,
grafting $g$-corolla on the $i$th-leaf of $f$-corolla,
one can define an operad composition $f\c_{i}g$,
namely, $f\c_{i}g$
is a tree whose vertices are decorated with $f$ and $g$.
For instance, if $f\in E_{3}$, $g,h\in E_{2}$, then
$(g\c_{1}f)\c_{4}h=$
\begin{center}
\includegraphics[scale=0.4]{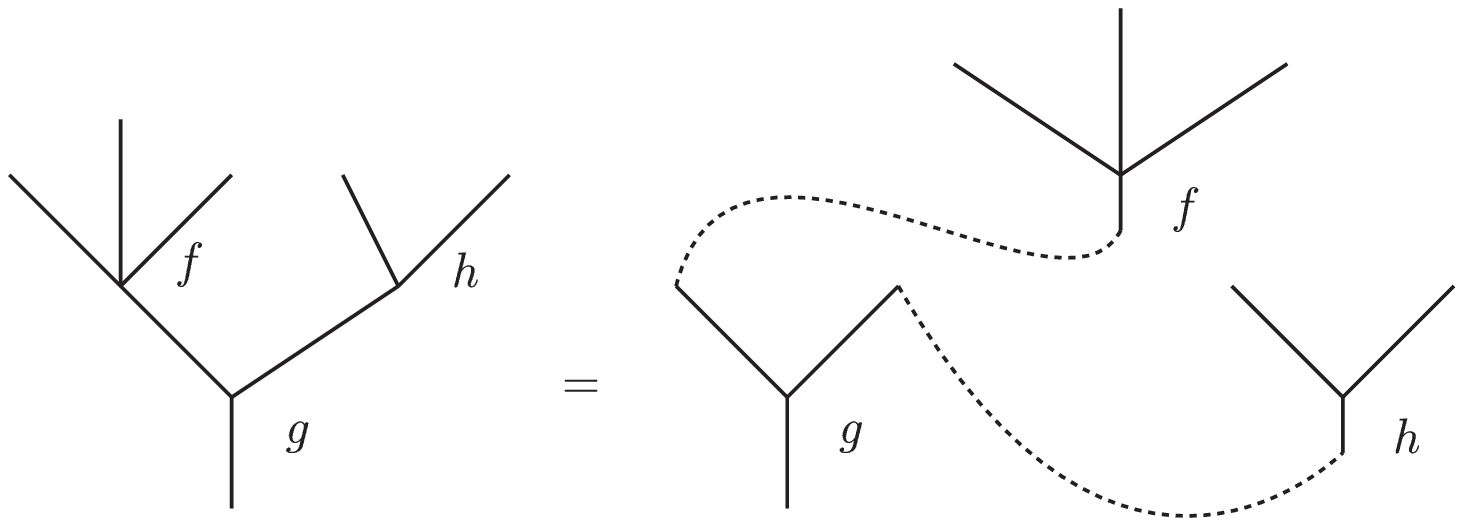}.
\end{center}
In this way, an operad,
which is denoted by $\F(E)$, is generated.
This is the free operad over the $S$-module.
Incidentally, the example $(g\c_{1}f)\c_{4}h$ is an element
in $\F(E)(5)$. 
\end{example}
Let $E$ be a (graded) vector space.
We identify $E$ with an $S$-module
such that $\{E_{2}=E, \ E_{n\neq 2}=0\}$.
We consider the free operad $\F(E)$.
A subspace $R\subset\F(E)(3)$ is called a quadratic relation,
if it is $S_{3}$-stable.
Let $R$ be a quadratic relation
and let $(R)$ be the generated operad ideal of the free operad.
The quotient operad $\P(E,R):=\F(E)/(R)$ is called
a \textbf{binary quadratic operad}. The above operads
$\Lie$, $Leib$ and $\Perm$ are binary quadratic operads.
One can show that $\F(E)(3)$ is generated by the
elements of the form:
$$
e\c_{1}e^{\p}(x_{i},x_{j},x_{k}),
$$
where $e,e^{\p}\in E$ and
$(i,j,k)\in\{(1,2,3),(3,1,2),(2,3,1)\}$.
We identify $e\c_{1}e^{\p}\cong e\ot e^{\p}$.
Then we obtain the following isomorphisms.
$$
\F(E)(3)\cong \bigoplus_{(i,j,k)}E\ot E\ot(i,j,k)
\cong3E\ot E.
$$
This gives the dimension relation for
the binary quadratic operad:
\begin{equation}\label{dimrel}
\dim \P(3)=3(\dim E)^{2}-\dim R.
\end{equation}
We will use this relation in the next section.
\medskip\\
\indent
We recall the shifted operads.
\begin{example}\label{shiftex}
(cf. \cite{GK} Definition 3.2.13; \cite{MSS} page 127)
The shifted operad of $\End(V)$ (which is denoted by $\s\End(V)$) is,
by definition, $\End(sV)$, where $s$ is the degree shifting operator
of degree $|s|=+1$.
An arbitrary element $\s f\in\s\End(V)(n)$ has the form,
$\s f=sf(s^{-1}\ot...\ot s^{-1})$.
One can regard $s^{-1}\ot...\ot s^{-1}$
as a one dimensional sign representation of $S_{n}$.
Thus we obtain
$$
sf(s^{-1}\ot...\ot s^{-1})\cong f[1-n]\ot{sgn(n)},
$$
where $f[1-n]$ is a copy of $f$ with degree $|f|+1-n$
and $sgn(n)$ is the one dimensional representation space.
\end{example}
Given an operad $\P$, the shifted operad $\s\P$
is defined by $\s\P:=\P[1-\cdot]\ot{sgn}$
in general,
where $sgn=\{sgn(n)\}_{n\in\mathbb{N}}$.
\section{Main results}
\subsection{Derived brackets}
Let $\P=\P(E,R)$ be a binary quadratic operad
and let $(A,d)$ be a dg $\P$-algebra.
We assume that $E$ is a homogeneous space
of degree zero (or even) and assume that $|d|=+1$ (or odd).
\begin{definition}
The derived brackets on $sA$ are defined to be
the following binary multiplications,
\begin{eqnarray*}
\ [sx_{1},sx_{2}]_{d}&:=&-(-1)^{|x_{1}|}s[dx_{1},x_{2}],\\
\ [sx_{1},sx_{2}]^{d}&:=&s[x_{1},dx_{2}],
\end{eqnarray*}
where $[\cdot,\cdot]$ is a $\P$-algebra multiplication
and $sx_{1},sx_{2}\in sA$, $x_{1},x_{2}\in A$.
\end{definition}
The derived brackets are given by
\begin{eqnarray*}
\ [\cdot,\cdot]_{d}&:=&s[\cdot,\cdot](d\ot 1)(s^{-1}\ot s^{-1}),\\
\ [\cdot,\cdot]^{d}&:=&-s[\cdot,\cdot](1\ot d)(s^{-1}\ot s^{-1}),
\end{eqnarray*}
on the level of the endomorphism operad.
If $[\cdot,\cdot]$ is (anti-)commutative, then $[\cdot,\cdot]^{d}$
is the (anti-)transposition of $[\cdot,\cdot]_{d}$, namely,
$$
[sx_{1},sx_{2}]^{d}=
(\pm 1)(-1)^{(|x_{1}|+1)(|x_{2}|+1)}[sx_{2},sx_{1}]_{d}.
$$
\noindent
Let $\mu$ be an $n$-monomial of $\P$-algebra.
It satisfies the derivation property,
$$
d\mu(x_{1},...,x_{n})=\sum(\pm)
\mu(x_{1},...,x_{i-1},dx_{i},x_{i+1},...,x_{n}),
$$
where $\pm$ is an appropriate sign.
Hence the monomials composed of the (nonshifted) derived brackets
are generated by the monomials of the form,
$$
\mu\c\d^{n}_{i}:=\mu(d\ot...\ot d\ot 1_{(ith)}\ot d\ot...\ot d).
$$
For instance, $d[\cdot,[d[\cdot,\cdot],d(\cdot)]]
=[d(\cdot),[[d(\cdot),\cdot],d(\cdot)]]
+[d(\cdot),[[\cdot,d(\cdot)],d(\cdot)]]=$
$$
=[\cdot,[[\cdot,\cdot],\cdot]]
(d\ot d\ot 1\ot d)
+[\cdot,[[\cdot,\cdot],\cdot]](d\ot 1\ot d\ot d).
$$
Our main result is the following.
\begin{theorem}\label{them1}
The algebra of the derived brackets is
a $\Perm\c\P$-algebra.
\end{theorem}
Vallette showed $\Perm\ot\P\cong\Perm\c\P$ in \cite{Val} Proposition 15
(See Appendix for the definition of Manin products).
So it suffices to show that $sA$ becomes
a $\Perm\ot\P$-algebra.
\begin{proof}
The monomials composed of the (shifted) derived brackets
are generated by the monomials of the form,
$$
\s(\mu\c\d^{n}_{i})=s\mu(d\ot...\ot d\ot 1\ot
d\ot...\ot d)(s^{-1}\ot...\ot s^{-1}).
$$
We remark that the degree of $\s(t\c\d^{n}_{i})$ is $0$.
We have to mix $d$ with $s^{-1}$, that is,
$$
\pm\s(\mu\c\d^{n}_{i})=s\mu(ds^{-1}\ot...\ot ds^{-1}\ot s^{-1}\ot
ds^{-1}...\ot ds^{-1}),
$$
where $\pm=(-1)^{n(n-1)/2+(i-1)}$.
We define maps,
$$
\Perm\ot\P\overset{1\ot rep}{\longrightarrow}
\Perm\ot\End(A)\overset{\phi}{\longrightarrow}\s\End(A),
$$
where $\phi$ is the collection of $\phi(n)$,
$$
\phi(n):e^{n}_{i}\ot \mu\mapsto\pm\s(\mu\c\d^{n}_{i}).
$$
The map $\phi$ preserves
the composition relations of $\Perm$ in Example \ref{exampleperm},
namely, the following relations hold.
\begin{eqnarray*}
\pm\s(\mu\c\d^{m}_{i})\c_{j}\pm\s(\mu^{\p}\c\d^{n}_{k})
=\left\{
\begin{array}{ll}
\pm\s(\mu\c_{j}\mu^{\p}\c\d^{m+n-1}_{i})& i<j \\
\pm\s(\mu\c_{j}\mu^{\p}\c\d^{m+n-1}_{i+k-1})& i=j \\
\pm\s(\mu\c_{j}\mu^{\p}\c\d^{m+n-1}_{i+n-1})& i>j.
\end{array}
\right.
\end{eqnarray*}
This implies that $\phi$ is an operad morphism.
Therefore $\phi\c(1\ot rep)$ is an operad representation.
The proof of the theorem is completed.
\end{proof}
The proof above shows that the composition
of the operad $\Perm$ corresponds to the derivation
property.
However the converse of this correspondence
is not clear.
In the next subsection we will show that
the quadratic relations of $\Perm\c\P$
are completely determined by
the properties that the derived brackets satisfy.
\medskip\\
\indent
We consider the derived brackets on abelian subalgebras.
Let $\g$ be a dg-Lie algebra.
If $\g^{\p}\subset\g$ is an abelian (trivial)
subalgebra of $\g$ (not dg subalgebra)
and if $s\g^{\p}$ is closed under the
derived bracket, then
the derived bracket is Lie on $s\g^{\p}$.
This proposition is one of the basic propositions
of the classical derived bracket construction (see \cite{Kos1}).
We should state an operadic version of this proposition.
\begin{proposition}\label{absubal}
Let $A$ be a dg $\P$-algebra
with a trivial subalgebra $A^{\p}$ (not dg subalgebra).
If $sA^{\p}$ is closed under the derived brackets,
then $sA^{\p}$ is a $\P$-algebra.
\end{proposition}
\begin{proof}
The commutative associative algebras are
obviously permutation-algebras.
Hence there exists an operadic projection
$pr:\Perm\to\Com$, where $\Com$
is the operad of commutative associative algebras.
Since $A^{\p}$ is a trivial subalgebra,
the derived brackets satisfy the following
relation on $sA^{\p}$.
$$
s[\cdot,\cdot](ds^{-1}\ot s^{-1})=s[\cdot,\cdot](s^{-1}\ot ds^{-1}).
$$
This implies that the representation of $\Perm\ot\P$ on $sA^{\p}$
factors through $\Com\ot\P$,
namely, the following diagram is commutative.
$$
\begin{CD}
\Perm\ot\P@>{rep}>>\End(sA^{\p}) \\
@V{pr\ot 1}VV @| \\
\Com\ot\P@>{rep^{\p}}>>\End(sA^{\p}),
\end{CD}
$$
where $rep^{\p}$ is the reduced representation.
It is known that $\Com(n)\cong\mathbb{K}$
for each $n\in\mathbb{N}$.
This gives $\Com\ot\P\cong\P$.
Therefore $sA^{\p}$ becomes a $\P$-algebra.
\end{proof}
\subsection{Free derived bracket}
We reconstruct $\Perm\c\P$ in terms of
the derived bracket construction.
First, we give a corollary of Theorem \ref{them1}.
\begin{corollary}
(Aguiar's derived brackets; cf. \cite{A})
Let $(A,\a)$ be a $\P$-algebra with
a linear endomorphism $\a:A\to A$.
Assume that $\a$ satisfies Rota's averaging relation:
$$
\a[\a x_{1},x_{2}]=[\a x_{1},\a x_{2}]=\a[x_{1},\a x_{2}].
$$
Aguiar defined the following modified brackets
(derived brackets),
$$
[\a x_{1},x_{2}] \ \ \text{and} \ \ [x_{1},\a x_{2}].
$$
The algebra of Aguiar's derived brackets is a $\Perm\c\P$-algebra.
\end{corollary}
\begin{proof}
The averaging relation corresponds to
the differential relation:
$d[dx_{1},x_{2}]=-[dx_{1},dx_{2}]=-d[x_{1},dx_{2}]$,
up to shift.
Therefore the corollary is shown in a similar manner.
\end{proof}
Let $\P$ be a binary quadratic operad over $(E,R)$.
Since $\dim\Perm(2)=2$, $\Perm(2)\ot\P(2)$ is
2-copies of $E$. We put
\begin{equation}\label{def2e}
\Perm(2)\ot\P(2):=\langle{
[\a x_{1},x_{2}], [x_{1},\a x_{2}] \ | \ [x_{1},x_{2}]\in E\ 
}\rangle,
\end{equation}
where $\alpha$ (left and right) is a label on copies
(not averaging operator).
We define an $S_{2}$-module structure on
$\Perm(2)\ot\P(2)$ by the natural manner:
$$
[\a x_{1},x_{2}](12)=[\a x_{2},x_{1}]=[x_{1},\a x_{2}]^{(12)},
$$
where $(12)\in S_{2}$
and $[\cdot,\cdot]^{(12)}\in E$ is the transposition of $[\cdot,\cdot]$.
One can regard $[\a x_{1},x_{2}]$ and $[x_{1},\a x_{2}]$
as binary corollas whose leaves are decorated with symbol $\a$
and whose vertices are decorated with $[x_{1},x_{2}]$:
\begin{center}
\includegraphics[scale=0.4]{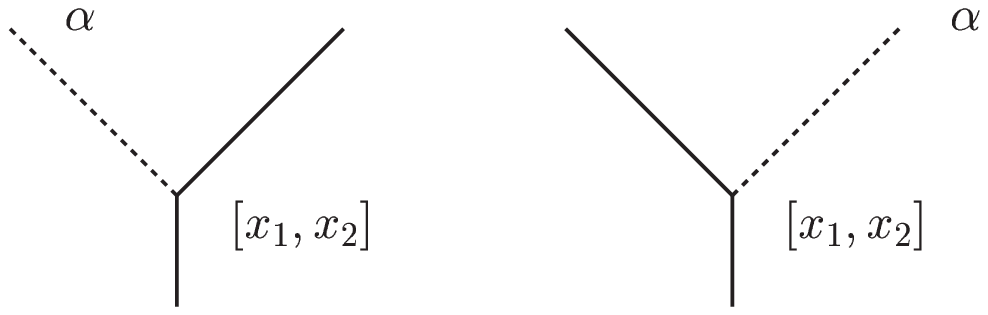},
\end{center}
where $\a$-edges are drew as broken lines.
We consider the free operad $\F(\Perm(2)\ot\P(2))$
and the quadratic relation below.
\begin{equation}\label{keyidentity}
[\a[\a x_{i},x_{j}],x_{k}]^{\p}
-[\a[x_{i},\a x_{j}],x_{k}]^{\p}\equiv 0,
\end{equation}
where $[\cdot,\cdot],[\cdot,\cdot]^{\p}\in E$
and $(i,j,k)\in\{(1,2,3),(3,1,2),(2,3,1)\}$.
Let $Ave$ denote the space of
relations generated by (\ref{keyidentity}).
We set the 3-copies of $\F(E)(3)$,
$$
3\F(E)(3):=\Big(\F(E)(3)\bar{\c}_{2}\a\bar{\c}_{3}\a\Big)
\oplus\Big(\F(E)(3)\bar{\c}_{3}\a\bar{\c}_{1}\a\Big)
\oplus\Big(\F(E)(3)\bar{\c}_{1}\a\bar{\c}_{2}\a\Big),
$$
where $\bar{\c}_{i}\a\bar{\c}_{j}\a$ are labels on copies.
We define a natural $S_{3}$-module structure on $3\F(E)$
such that a map $\theta$ below is equivariant.
The map $\theta$ is defined to be
a surjection $\theta:\F(\Perm(2)\ot\P(2))(3)\to 3\F(E)(3)$ by
\begin{eqnarray*}
\ [\a[\a x_{i},x_{j}],x_{k}]^{\p}&\mapsto&
[[x_{i},x_{j}],x_{k}]^{\p}\bar{\c}_{i}\a\bar{\c}_{j}\a,\\
\ [\a[x_{i},\a x_{j}],x_{k}]^{\p}&\mapsto&
[[x_{i},x_{j}],x_{k}]^{\p}\bar{\c}_{i}\a\bar{\c}_{j}\a,\\
\ [[\a x_{i},x_{j}],\a x_{k}]^{\p}&\mapsto&
[[x_{i},x_{j}],x_{k}]^{\p}\bar{\c}_{k}\a\bar{\c}_{i}\a,\\
\ [[x_{i},\a x_{j}],\a x_{k}]^{\p}&\mapsto&
[[x_{i},x_{j}],x_{k}]^{\p}\bar{\c}_{j}\a\bar{\c}_{k}\a,
\end{eqnarray*}
where $(i,j,k)\in\{(1,2,3),(3,1,2),(2,3,1)\}$.
The relation $Ave$ is the kernel of $\theta$.
\begin{lemma}
The dimension of $\theta^{-1}(3R)$
is equal to $3(\dim E)^{2}+3\dim R$, and
it coincides with the dimension
of the space of quadratic relations of $\Perm\c\P$.
\end{lemma}
\begin{proof}
One can easily check that $\dim Ave=3(\dim E)^{2}$.
This gives the dimension relation
$\dim \theta^{-1}(3R)=3(\dim E)^{2}+3\dim R$.
Since $\Perm\c\P\cong\Perm\ot\P$
and $\dim\Perm(3)=3$,
from (\ref{dimrel}), we obtain
$$
\dim R_{\Perm\c\P}=
3(\dim E)^{2}+3\dim R,
$$
where $R_{\Perm\c\P}$ is the relation of $\Perm\c\P$.
\end{proof}
Since $\theta$ is equivariant, $\theta^{-1}(3R)$
is a quadratic relation.
The main result of this subsection is as follows.
\begin{proposition}\label{keyalpha}
$\F(\Perm(2)\ot\P(2))/\big(\theta^{-1}(3R)\big)
\cong\Perm\ot\P(\cong\Perm\c\P)$.
\end{proposition}
\begin{proof}
We define a bijection
$\xi:\Perm(2)\ot\P(2)\to(\Perm\ot\P)(2)$ by
\begin{eqnarray*}
\ [\a x_{1},x_{2}]&\mapsto&e^{2}_{2}\ot[x_{1},x_{2}],\\
\ [x_{1},\a x_{2}]&\mapsto&e^{2}_{1}\ot[x_{1},x_{2}].
\end{eqnarray*}
Claim. This bijection is equivariant.\\
Hence it induces an operadic isomorphism
$\xi:\F(\Perm(2)\ot\P(2))\cong\F\big((\Perm\ot\P)(2)\big)$.
We show that this isomorphism induces
the isomorphism of the proposition.
It suffices to show that the image of $\theta^{-1}(3R)$
by $\xi$ vanishes on $\Perm\ot\P$.
The relation $Ave$ vanishes in $\Perm\ot\P$, because
\begin{eqnarray*}
\ [\a[\a x_{1},x_{2}],x_{3}]^{\p}&\overset{\xi}{\mapsto}&
e^{2}_{2}\ot[x_{1},x_{2}]^{\p}\c_{1}e^{2}_{2}\ot[x_{1},x_{2}]
\overset{\pi}{\mapsto}
e^{3}_{3}\ot[[x_{1},x_{2}],x_{3}]^{\p},\\
\ [\a[x_{1},\a x_{2}],x_{3}]^{\p}&\overset{\xi}{\mapsto}&
e^{2}_{2}\ot[x_{1},x_{2}]^{\p}\c_{1}e^{2}_{1}\ot[x_{1},x_{2}]
\overset{\pi}{\mapsto}
e^{3}_{3}\ot[[x_{1},x_{2}],x_{3}]^{\p},
\end{eqnarray*}
where $\pi$ is the projection onto $\Perm\ot\P$.\\
Claim. By a direct computation, one can show that
$$
\pi\c\xi\Big(
\theta^{-1}
\big(
[[x_{a},x_{b}],x_{c}]^{\p}\bar{\c}_{j}\a\bar{\c}_{k}\a
\big)\Big)=
e^{3}_{i}\ot[[x_{a},x_{b}],x_{c}]^{\p},
$$
where $(a,b,c),(i,j,k)\in\{(1,2,3),(3,1,2),(2,3,1)\}$.\\
Therefore we obtain
$$
\pi\c\xi:
\theta^{-1}(R\bar{\c}_{j}\a\bar{\c}_{k}\a)\mapsto
e^{3}_{i}\ot R.
$$
The image $e^{3}_{i}\ot R$ is zero on $\Perm\ot\P$,
because $R$ is the relation of $\P$.
\end{proof}
We will use this proposition in the next section.
The relation $\theta^{-1}(3R)$ is the basic properties of
the derived brackets.
Hence the operad $\Perm\c\P$ is completely determined
by the properties that the derived brackets satisfy.
\subsection{An example}
We consider the derived brackets on dg Poisson algebras.
Let $(P,\cdot,\{\cdot,\cdot\},d)$ be a dg Poisson algebra.
We assume that the degrees of the multiplications
are both zero (or even).
We denote by $\Poiss$
the operad of Poisson algebras.
The derived bracket/product are defined by
\begin{eqnarray*}
[sx,sy]_{d}&:=&-(-1)^{|x|}s\{dx,y\},\\
sx*_{d}sy&:=&-(-1)^{|x|}s\big((dx)y\big).
\end{eqnarray*}
From the symmetry of Poisson structures,
we have $[1,2]^{d}=-[2,1]_{d}$ and $1*^{d}2=2*_{d}1$.
Then $sP$ becomes a $\Perm\c\Poiss$-algebra.
The derived bracket is Leibniz (I) and the derived product
$*_{d}$ is perm (II),
and they are satisfying the three conditions below.
\begin{eqnarray}
\label{pp1}
[sx,sy*_{d}sz]_{d}&=&
[sx,sy]_{d}*_{d}sz+(-1)^{(|x|+1)(|y|+1)}sy*_{d}[sx,sz]_{d},\\
\label{pp2}
[sx*_{d}sy,sz]_{d}&=&
sx*_{d}[sy,sz]_{d}+(-1)^{(|x|+1)(|y|+1)}sy*_{d}[sx,sz]_{d},\\
\label{pp3}
[sx,sy]_{d}*_{d}sz&=&-(-1)^{(|x|+1)(|y|+1)}[sy,sx]_{d}*_{d}sz.
\end{eqnarray}
Aguiar (\cite{A}) introduced two new type of algebras,
namely, {\em prePoisson algebras} and {\em dual-prePoisson algebras}.
The latter is the Koszul dual of the former.
A dual-prePoisson algebra is defined to be
an algebra equipped with
two multiplications, $[,]$ and $*$, satisfying
(I), (II), (\ref{pp1}), (\ref{pp2}) and (\ref{pp3}).
Therefore the algebra of the derived bracket/product on
a Poisson algebra is a dual-prePoisson algebra.
From this observation, we obtain
\begin{proposition}
$pre\Poiss^{!}\cong\Perm\c\Poiss$.
\end{proposition}
\begin{proof}
The relations (I), (II), (\ref{pp1}),
(\ref{pp2}) and (\ref{pp3})
are subrelations of $\theta^{-1}(3R_{\Poiss})$,
because they are the basic properties of the derived
bracket/product.
The dimension of the space of quadratic relations
of $pre\Poiss^{!}$ is $30$:
$$
30=(6+9)+(6+6+3),
$$
where $(6+9)$ is the dimension of the space
generated by (I) and (II),
$(6+6+3)$ is the dimension of the space of
generated by (\ref{pp1}), (\ref{pp2}) and (\ref{pp3}).
On the other hand, since $\dim R_{\Poiss}=6$,
$\dim\theta^{-1}(3R_{\Poiss})=30$.
The proof is completed.
\end{proof}
Since $\Poiss\cong\Poiss^{!}$ and $pre\Lie\cong\Perm^{!}$,
we have $pre\Poiss\cong pre\Lie\bullet\Poiss$.
\section{Koszul duality theory}
It was shown in \cite{GK,GK2} Theorem 2.2.6
that the Koszul dual
of the Manin white product $\P_{1}\c\P_{2}$
is the Manin black product, $\P_{1}^{!}\bullet\P_{2}^{!}$
(see also Appendix below).
It is known that the Koszul dual of the operad $\Perm$
is the operad of preLie algebras, $pre\Lie$ (cf. \cite{Chap,Chap2}).
Hence we have
$$
(\Perm\c\P)^{!}=pre\Lie\bullet\P^{!},
$$
which implies that the operad $pre\Lie$ is closely related with
the derived bracket construction.
In this section, we discuss a Koszul duality
for the derived bracket construction.
\medskip\\
\indent
Let $\P$ be a binary quadratic operad
and let $A$ be a $\P$-algebra equipped with
multiplications $(*_{1},...,*_{n})$.
By definition, a Rota-Baxter operator (with weight zero)
on $A$ is a linear endomorphism $\b:A\to A$ satisfying
the Rota-Baxter identity:
$$
\b(x)*_{i}\b(y)=\b(\b(x)*_{i}y+x*_{i}\b(y)),
$$
for any $x,y\in A$ and for any $i$ (cf. \cite{A} \cite{AL}).
For example, the ordinary integral operator
$$
\b(f)(x):=\int^{x}_{0}f(t)dt
$$
is a Rota-Baxter operator on the algebra of functions $C^{0}([0,1])$.
In this sense, the Rota-Baxter operators are considered as
a kind of dual-operators of differential operators.
It is inferred that the integral/Rota-Baxter operators
are the Koszul dual of the differential/averaging operators
from some known observations (see \cite{A} \cite{AL}).
In Theorem \ref{them2} below,
we give definite evidence that this statement is true.
\begin{definition}
Let $\P=\P(E,R)$ be an arbitrary binary
quadratic operad and let $(A,\b)$ be a $\P$-algebra
with a Rota-Baxter operator.
We call the following multiplications
the \textbf{dual-derived products}.
$$
\b(x)*_{i}y \ \ \text{and} \ \ x*_{i}\b(y).
$$
\end{definition}
The original idea of the dual-derived products was given
by Aguiar in \cite{A}.
\begin{theorem}\label{them2}
Let $\P=\P(E,R)$ be an arbitrary binary quadratic operad
and let $\P^{!}(E^{\vee},R^{\bot})$ the Koszul dual
and let $(A,\b)$ a $\P^{!}$-algebra with
a Rota-Baxter operator.
The algebra of the dual-derived products
is a $pre\Lie\bullet\P^{!}$-algebra.
\end{theorem}
\begin{proof}
First we study the dual of Proposition \ref{keyalpha}.
We recall (\ref{def2e}) in Section 3.2.
In the same way, we define 2-copies of $E^{\vee}$:
$$
pre\Lie(2)\ot\P^{!}(2):=\langle
\b x_{1}*x_{2},x_{1}*\b x_{2} \ | \ x_{1}*x_{2}\in E^{\vee} 
\rangle,
$$
where $\b$ is a label on copies
(not Rota-Baxter operator).
The $S_{2}$-module structure on $pre\Lie(2)\ot\P^{!}(2)$ is
defined by the natural manner:
$$
(\b x_{1}*x_{2})(12)=\b x_{2}*x_{1}=x_{1}*^{(12)}\b x_{2},
$$
where $*^{(12)}\in E^{\vee}$ is the transposition of $*$.
The operadic duality of $pre\Lie(2)\ot\P^{!}(2)$
and $\Perm(2)\ot\P(2)$ is well-defined by
\begin{eqnarray*}
<\b x_{1}*x_{2},[\a x_{1},x_{2}]>&:=&
<x_{1}*x_{2},[x_{1},x_{2}]>,\\
<x_{1}*\b x_{2},[x_{1},\a x_{2}]>&:=&
<x_{1}*x_{2},[x_{1},x_{2}]>,\\
\text{otherwise}&:=&0,
\end{eqnarray*}
where the pairing of the right-hand side
is the duality of $E^{\vee}$ and $E$.
We set the 3-copies of $\F(E^{\vee})(3)$, like $3\F(E)(3)$,
$$
3\F(E^{\vee})(3):=
\Big(\F(E^{\vee})(3)\bar{\c}_{2}\b\bar{\c}_{3}\b\Big)
\oplus\Big(\F(E^{\vee})(3)\bar{\c}_{3}\b\bar{\c}_{1}\b\Big)
\oplus\Big(\F(E^{\vee})(3)\bar{\c}_{1}\b\bar{\c}_{2}\b\Big).
$$
The duality of $3\F(E)(3)$ and $3\F(E^{\vee})(3)$
is naturally defined.
We recall the projection $\theta$ in Section 3.2.
We consider the dual map of $\theta$:
$$
\F(pre\Lie(2)\ot\P^{!}(2))(3)
\leftarrow 3\F(E^{\vee})(3):\theta^{\vee}.
$$
The image of $3R^{\bot}$, $\theta^{\vee}(3R^{\bot})$,
is a quadratic relation in $\F(2E^{\vee})(3)$.
As a corollary of Proposition \ref{keyalpha}, we obtain
\begin{lemma}
$\F(pre\Lie(2)\ot\P^{!}(2))/\big(\theta^{\vee}(3R^{\bot})\big)
\cong pre\Lie\bullet\P^{!}$.
\end{lemma}
By a direct computation, we obtain
\begin{eqnarray*}
\label{bet2}\theta^{\vee}\Big(
(x_{i}*x_{j})*^{\p}x_{k}\bar{\c}_{j}\b\bar{\c}_{k}\b
\Big)
&=&(x_{i}*\b x_{j})*^{\p}\b x_{k},\\
\label{bet3}\theta^{\vee}\Big(
(x_{i}*x_{j})*^{\p}x_{k}\bar{\c}_{k}\b\bar{\c}_{i}\b
\Big)
&=&(\b x_{i}*x_{j})*^{\p}\b x_{k},
\end{eqnarray*}
and from the condition, $<\theta^{\vee}(-),Ave>=0$,
we get the {\em formal} Rota-Baxter identity,
\begin{equation*}\label{bet1}
\theta^{\vee}\Big(
(x_{i}*x_{j})*^{\p}x_{k}\bar{\c}_{i}\b\bar{\c}_{j}\b
\Big)
=\b(\b x_{i}*x_{j})*^{\p}x_{k}+\b(x_{i}*\b x_{j})*^{\p}x_{k}.
\end{equation*} 
Now we prove the theorem.
Let $r\in R^{\bot}$ be a quadratic relation of $\P^{!}$.
Then one can write
$$
r=\sum_{(*,*^{\p})}C_{1}(x_{1}*x_{2})*^{\p}x_{3}+
C_{2}(x_{3}*x_{1})*^{\p}x_{2}+C_{3}(x_{2}*x_{3})*^{\p}x_{1},
$$
where $C_{1},C_{2},C_{3}$ are structure constants.
We obtain
\begin{multline*}
\theta^{\vee}(r\bar{\c}_{1}\b\bar{\c}_{2}\b)=
\sum_{(*,*^{\p})}C_{1}\b(\b x_{1}*x_{2})*^{\p}x_{3}+
C_{1}\b(x_{1}*\b x_{2})*^{\p}x_{3}+\\
C_{2}(x_{3}*\b x_{1})*^{\p}\b x_{2}+C_{3}(\b x_{2}*x_{3})*^{\p}\b x_{1}.
\end{multline*}
If $\b$ is a {\em real} Rota-Baxter operator, that is,
if the relation is represented on the algebra of
the dual-derived products, then
\begin{eqnarray*}
\theta^{\vee}(r\bar{\c}_{1}\b\bar{\c}_{2}\b)&\overset{rep}{\to}&
\sum_{(*,*^{\p})}C_{1}(\b x_{1}*\b x_{2})*^{\p}x_{3}+
C_{2}(x_{3}*\b x_{1})*^{\p}\b x_{2}+C_{3}(\b x_{2}*x_{3})*^{\p}\b x_{1}\\
&=&r\c_{1}\b\c_{2}\b=0 \ \ \text{(on $A$)}.
\end{eqnarray*}
This holds for each
$\theta^{\vee}(r\bar{\c}_{i}\b\bar{\c}_{j}\b)$,
$(i,j)\in\{(1,2),(2,3),(3,1)\}$.
The proof of the theorem is completed.
\end{proof}
Since $\Lie\bullet\P\cong\P$,
we call $pre\Lie\bullet\P$-algebras simply $pre\P$-algebras.
\begin{corollary}(\cite{Val} Proposition 22).
Let $A$ be a $pre\P$-algebra equipped with
multiplications $\b x_{1}*_{i}x_{2}$ and $x_{1}*_{i}\b x_{2}$,
where $i\in\{1,...,\dim\P(2)\}$.
We define a commutator associated with $*_{i}$,
$$
\{x_{1},x_{2}\}_{i}:=\b x_{1}*_{i}x_{2}+x_{1}*_{i}\b x_{2}.
$$
The algebra of the commutators is a $\P$-algebra.
\end{corollary}
\begin{proof}
Let $r=r(*_{1},*_{2},...)$ be a $\P$-algebra relation
composed of $\P$-algebra multiplications.
We have
$$
\sum_{(i,j)\in\{(1,2),(2,3),(3,1)\}}
\theta^{\vee}(r\bar{\c}_{i}\b\bar{\c}_{j}\b)
=r(\{\cdot,\cdot\}_{1},\{\cdot,\cdot\}_{2},...)=0.
$$
\end{proof}
The corollary above is considered to be a dual of
Proposition \ref{absubal},
because Proposition \ref{absubal} relies on
the projection $\Perm\to\Com$ and the above corollary
relies on the dual map $pre\Lie\leftarrow\Lie$.
\section{Appendix}
\textbf{A1. White products}.
Let $\P_{i}$ be a binary quadratic operad over
$(E_{i},R_{i})$, $i\in\{1,2\}$.
\begin{definition}
(White products
\cite{GK,GK2} Section 2.2;
\cite{Val} Section 3.2)
$$
\P_{1}\c\P_{2}:=\frac{\F(E_{1}\ot E_{2})}
{\Phi^{-1}\big(R_{1}\ot\F(E_{2})(3)+\F(E_{1})(3)\ot R_{2}\big)},
$$
where $\Phi$ is the unique operad morphism
in the universal diagram below.
$$
\begin{CD}
E_{1}\ot E_{2}@>{\text{universal arrow}}>>\F(E_{1}\ot E_{2}) \\
@| @V{\Phi}VV \\
E_{1}\ot E_{2}@>>>\F(E_{1})\ot\F(E_{2}).
\end{CD}
$$
\end{definition}
\noindent
\textbf{A2. Operadic dual spaces}.
Let $E$ be an $S_{n}$-module.
The operadic dual of $E$,
which is denoted by $E^{\vee}$,
is a dual space of $E$
whose duality is defined by the pairing:
$$
<e,f>=sgn(\sigma)<e\sigma,f\sigma>,
$$
where $e\in E$, $f\in E^{\vee}$ and $\sigma\in S_{n}$.
We consider the case of $n=2$.
Let $e_{i}$ be a noncommutative
multiplication and let $e^{(12)}_{i}$ the transposition
of $e_{i}$ and let $e^{\pm}_{j}$ be
commutative $(+)$, anti-commutative $(-)$
multiplications, respectively.
One can regard the duality as the pseudo-Euclidean metric
defined by
\begin{eqnarray*}
<e_{i},e_{j}>&=&\delta_{ij},\\
<e^{(12)}_{i},e^{(12)}_{j}>&=&-\delta_{ij},\\
<e^{+}_{i},e^{-}_{j}>&=&\delta_{ij},\\
<e^{-}_{i},e^{+}_{j}>&=&\delta_{ij},
\end{eqnarray*}
and all others $0$.
The invariancy of the pairing above
is built in the metric.
\medskip\\
\noindent
\textbf{A3. Koszul dual}.
One can identify
$\F(E)(3)$ with the 3-copies $3(E\ot E)$.
This isomorphism is defined by
$$
e\c_{1}e^{\p}(x_{i},x_{j},x_{k})\cong e\ot e^{\p}\ot(i,j,k),
$$
where $(i,j,k)\in\{(1,2,3),(3,1,2),(2,3,1)\}$.
For instance,
$$
[[x_{i},x_{k}],x_{j}]^{\p}=[[x_{k},x_{i}]^{(12)},x_{j}]^{\p}
\cong[x_{1},x_{2}]^{\p}\ot[x_{1},x_{2}]^{(12)}\ot(k,i,j).
$$
The duality of $E$ and $E^{\vee}$ is naturally extended
to the duality of $\F(E)(3)$ and $\F(E^{\vee})(3)$,
via the isomorphism. More explicitly,
the operadic duality of $\F(E)(3)$ and $\F(E^{\vee})(3)$
is well-defined by
$$
<e\ot e^{\p}\ot(i,j,k),
f\ot f^{\p}\ot(a,b,c)>
:=<e,f><e^{\p},f^{\p}>\delta_{ia}\delta_{jb}\delta_{kc}.
$$
This pairing is regarded as a metric.
\begin{definition}
Let $R$ be an $S_{3}$-subspace of $\F(E)(3)$, i.e.,
quadratic relation
and let $R^{\bot}\subset\F(E^{\vee})(3)$
the orthogonal space of $R$
with respect to the operadic duality.
The Koszul dual of $\P:=\F(E)/(R)$
is defined to be the operad
$$
\P^{!}:=\F(E^{\vee})/(R^{\bot}).
$$
\end{definition}
Since $R^{\bot\bot}\cong R$, we have $\P^{!!}\cong\P$.
\medskip\\
\noindent
\textbf{A4. Black products}.
We consider the dual map of $\Phi$ on the third component.
$$
\F(E_{1}^{\vee}\ot E_{2}^{\vee})(3)\leftarrow
\F(E_{1}^{\vee})(3)\ot\F(E_{2}^{\vee})(3):\Psi,
$$
where $\Psi=\Phi(3)^{\vee}$.
\begin{definition}
(Black product
\cite{GK,GK2} Section 2.2;
\cite{Val} Section 4.3)
$$
\P^{!}_{1}\bullet\P^{!}_{2}:=\frac{\F(E^{\vee}_{1}\ot E^{\vee}_{2})}
{\Psi\big(R^{\bot}_{1}\ot\F(E^{\vee}_{2})(3)
\cap\F(E^{\vee}_{1})(3)\ot R^{\bot}_{2}\big)}.
$$
\end{definition}
Since $\P^{!!}\cong\P$, the black product
is defined for any binary quadratic operads.
It is obvious that
$(\P_{1}\c\P_{2})^{!}\cong\P^{!}_{1}\bullet\P^{!}_{2}$,
because
$\Psi\big(R^{\bot}_{1}\ot\F(E^{\vee}_{2})(3)
\cap\F(E^{\vee}_{1})(3)\ot R^{\bot}_{2}\big)$
is the orthogonal space of
$\Phi^{-1}\big(R_{1}\ot\F(E_{2})(3)+\F(E_{1})(3)\ot R_{2}\big)$.

\noindent
Science University of Tokyo.\\
Wakamiya 26 Shinjuku Tokyo Japan.\\
e-mail: K\underline{ }Uchino[at]oct.rikadai.jp

\end{document}